\documentclass[a4paper,11pt]{amsart}

\usepackage[a4paper]{geometry}
\usepackage[OT2,T1]{fontenc}

\usepackage{lipsum} 

\usepackage[english]{babel}
\usepackage{amsmath,amssymb}
\usepackage{enumerate}
\usepackage{array}
\usepackage{mathrsfs}
\usepackage[utf8]{inputenc}
\usepackage[T1]{fontenc}
\usepackage{version}
\usepackage{lscape}
\usepackage[all]{xy}
\usepackage{datetime}
\usepackage{todonotes}
\usepackage{graphicx}
\usepackage{multirow}
\usepackage{hyperref}
\usepackage{xcolor}

\newtheorem{definition}{Definition}

\newtheorem{thm}{Theorem}[section]
\newtheorem{lm}[thm]{Lemma}
\newtheorem{propo}[thm]{Proposition}
\newtheorem{coro}[thm]{Corollary}
\newtheorem{conjecture}[thm]{Conjecture}
\newtheorem{rem}{Remark}

\def\Z{\mathbb{Z}}
\def\N{\mathbb{N}}
\def\Q{\mathbb{Q}}
\def\R{\mathbb{R}}
\def\E{\mathscr{E}}
\def\W{\mathscr{W}}
\def\P{\mathbb{P}}

\newcommand{\cA}{\mathcal{A}}
\newcommand{\A}{{\rm A}}

\newcommand{\bZ}{\mathbb{Z}}

\newcommand{\bR}{\mathbb{R}}

\title{Root number in integer parameter families of elliptic curves}
\author{Julie Desjardins}

\begin{document}

\maketitle

\begin{abstract}
In a previous article \cite{Desjardins1}, the author proves that the value of the root number varies in a non-isotrivial family of elliptic curves indexed by one parameter $t$ running through $\Q$. However, a well-known example of Washington has root number $-1$ for every fiber when $t$ runs through $\Z$. Such examples are rare since, as proven in this paper, the root number of the integer fibers varies for a large class of families of elliptic curves. This result depends on the squarefree conjecture and Chowla's conjecture, and is unconditional in many cases.

Dans un article pr\'ec\'edent \cite{Desjardins1}, l'autrice d\'emontre que la valeur du signe varie parmi une famille non isotriviale de courbes elliptiques index\'ee par un param\`etre rationnel $t$. Cependant, un exemple bien connu de Washington a comme signe $-1$ pour toute fibre lorsque l'on se restreint aux valeurs enti\`eres de $t$. Pareils exemples sont rares: ce papier d\'emontre en effet que le signe des fibres enti\`eres varie pour une large classe de familles de courbes elliptiques. Ce r\'esultat d\'epend de la Conjecture du Crible sans Facteur Carr\'e et de la Conjecture de Chowla, et est inconditionnel dans plusieurs cas.
\end{abstract}

\section{Introduction}

A \emph{family of elliptic curves} $\E$ over a number field $K$ (or equivalently an \emph{elliptic surface with base $\P^1$}) is the family given by a Weierstrass equation 
$$\E:Y^2=X^3-27c_4(T)X+54c_6(T),$$
where $c_4,c_6\in\Z[T]$ are such that the discriminant $\Delta_\E(t)=\Delta(t)=\frac{c_4(t)^2-c_6(t)^3}{1728}$ is non-zero (except for finitely many values of $t$). We call the family \emph{isotrivial} if every fiber is a twist of one another (i.e. if the function $j: t\mapsto j(t)$ is constant), and \emph{non-isotrivial} otherwise.

A family of elliptic curves $\E$ can as well be seen as an elliptic curve over $\Q(T)$, and so it makes sense to speak about the \emph{generic rank} of $\E$, as well as the \emph{reduction of $\E$ at a place $v$ of $\Q(T)$}, whose type is described by a Kodaira symbol. Let us define $\mathscr{M}$ to be the set of places of multiplicative reduction and $\mathscr{B}$ the set of bad places that are not of Kodaira type $I_0^*$. Each place of $\Q(T)$ is either $-\deg$, or is said to be \emph{finite}, and in that case it is associated to an irreducible polynomial $P$ (that we can suppose with integer coefficients and primitive). 
Let us call \emph{insipid} (an irreducible polynomial $P$ associated to) a place of type \begin{itemize}
\item $I_0^*$
\item $II$, $II^*$, $IV$ or $IV^*$, and $$\mu_3\subseteq \Q[T]/P(T)\text{,$\quad$ where $\mu_3$ is the group of third root of unity;}$$
\item $III$ or $III^*$, and $$\mu_4\subseteq \Q[T]/P(T)\text{,$\quad$ where $\mu_4$ is the group of fourth root of unity.}$$
\end{itemize}

In this article, we are interested in the \emph{integer fibers} of a family of elliptic curves, meaning the curves $\E_t:=\E(t)$ such that $t\in\Z$, and in particular, we want to find more about the distribution of their rank. We will use a useful substitute called the \emph{root number} $W(\E_t)$ that is easier to compute.

Given an elliptic curve $E$ over $\Q$, the root number of $E$ is the product of the \emph{local root numbers}:
\begin{equation}\label{defrootnumber}W(E)=\prod_{p\text{ place of $\Q$}}{W_p(E)},\end{equation}
where $W_p(E)\in\{\pm1\}$ is defined in terms of the epsilon factors of the Weil-Deligne representations of $\Q_p$ and equals to $+1$ except for a finite number of $p$. For a more detailed definition of these local root numbers, we refer to \cite{Del} and \cite{Tatefibre}, and for explicit formulas we use the work of Rohrlich \cite{Rohr}, as well as (when $p\not=2,3$) the tables of Halberstadt \cite{Halb} completed by Rizzo \cite{Rizz}.

Let $L(E,s)$ be the associated $L$-function. Then by the modularity theorem, $L$ admits the following functional equation 
$$\mathscr{N}_E^{(2-s)/2}(2\pi)^{s- 2}\Gamma(2-s)L(E,2-s)=W(E)\mathscr{N}_E^{s/2}(2\pi)^{-s}\Gamma(s)L(E,s)$$
where $W(E)\in\{-1,+1\}$, the sign of the functional equation, is equal to the root number of $E$. For a general number field $K$, the existence of such a functional equation is not guaranteed, so we rely only on the definition (\ref{defrootnumber}). This equivalent definition leads to the equality: $W(E)=(-1)^{r_{an}(E)},$ where the \emph{analytic rank} is $r_{an}(E)=ord_{s=1}(L(E,s))$. By (a weaker version of) the Birch and Swinnerton-Dyer conjecture, we have $$r_{an}(E)\equiv r(E)\mod 2,$$ where $r(E)$ is the geometric rank of $E$. This justifies the following statement known as the \textbf{parity conjecture} $$W(E)=(-1)^{r(E)}.$$

Given $\E$ a family of elliptic curves over $\Q$ and $A\subseteq\Q$ (in particular $A=\Z$ or $\Q$), we consider the sets $W_+(\E,A)$ and $W_-(\E,A)$ given by \[W_\pm(\E,A)=\{t\in A: \E_t\text{ is an elliptic curve and }W(\E_t)=\pm1\}.\]

The author proves in \cite{Desjardins1} that for a non-isotrivial family of elliptic curves whose discriminant respects some analytic number theory conjectures\footnote{Those conjectures (Chowla's conjecture and the squarefree conjecture) are assumed in their homogeneous form in the main result of \cite{Desjardins1}.} true for low degree, we have
$$\#W_\pm(\E,\Q)=\infty.$$
Beforehand, Manduchi already proved Theorem \ref{thmgeneral} in the particular case where the family considered admits a place of multiplicative reduction \cite{Manduchi}, and Helfgott studied the average root number over $\Q$ and $\Z$ in the preprint \cite{Helfgott}.

\subsection{Main results}
 
In this paper, we prove analogues of the results of \cite{Desjardins1} (on the variation of the root number in families of elliptic curves in one parameter running through $\Q$) when restricted to integer fibers.
We show the following theorem.

\begin{thm}\label{thmgeneral}
Let $\E$ be a family of elliptic curves over $\Q$. Suppose that there exists at least one irreducible primitive polynomial $P_o(T)\in\Z[T]$ associated to either a place that is not insipid for $\E$ or to a place of type  $I_0^*$ and $\deg P_o$ is odd.

Suppose moreover that 
\begin{enumerate}[(i)]
\item $M_\E=\prod_{v}{P_v}$ (where the product runs through the places of multiplicative reduction) respects Chowla's conjecture (\ref{CC}) or $M_\E=1$; 
\item $B_\E=\prod_{v}{P_v}$ (where the product runs through the non-insipid places) respects the squarefree conjecture (\ref{sfc1}).
\end{enumerate}
Then the sets $W_+(\E,\Z)$ and $W_-(\E,\Z)$ are both infinite.
\end{thm}

Chowla's conjecture (stated further on) is known to hold for linear polynomials $P\in\Z[T]$. This conjecture gives an estimation of the proportion of values $P(t)$ with a given parity of number of prime factors.
The Squarefree conjecture (also stated further on), is known to hold when every irreducible factor of $\Delta_\E$ has degree at most 3. This conjecture gives an estimation of the proportion of values $P(t)$ that are squarefree (or "almost squarefree"). 
Our result is unconditional in the following cases:

\begin{coro}\label{corogeneral}
Let $\E$ be a family of elliptic curves with at least one finite place that is not insipid, or is of type $I_0^*$ associated to a polynomial of odd degree.

Suppose moreover that $\deg M_\E\leq1$ and that $\deg P\leq3$ for any non-insipid places. 
Then the sets $W_+(\E,\Z)$ and $W_-(\E,\Z)$ are both infinite.
\end{coro}

The proof of Theorem \ref{thmgeneral} is based on a one variable analogue of the "Squarefree Liouville Sieve" developped by the author in a previous paper \cite{Desjardins1} where she uses it to prove the variation of the root number on the rational fibers of a non-isotrivial family of elliptic curves. This sieve was inspired by the various squarefree sieves obtained consecutively by Hooley \cite{Hool}, Gouv\^ea and Mazur \cite{GM}, Rohrlich \cite{Rohr} and Varilly-Alvarado \cite{VA}.
The use of those sieves for similar purpose was already present in \cite{Manduchi} and in \cite{VA}.

In the preprint \cite{Helfgott}, Helfgott gives formulas for the computation of the average root number of the rational fibers of a non-isotrivial family of elliptic curves, and proves that it is equal to 0 whenever there is a place of multiplicative reduction. In \cite{HCC}, Helfgott together with Conrad and Conrad state without proving, that the average root number without multiplicative reduction is strictly contained between $-1$ and $1$ (the proof of this fact appears though in the author's Ph.D dissertation). The average root number on the integer fibers of a family of elliptic curves with coefficients of small degrees are also discussed in \cite{BDD} and \cite{Chinis}.

The well-known Washington's family $$\mathscr{V}: y^2=x^3+Tx^2-(T+3)x+1$$ introduced in \cite{Washington} has the properties that
\begin{itemize}
\item (\cite{Rizz}) for all $t\in\Z$, one has $W(\mathscr{V}_t)=-1$ (i.e. $W_+(\mathscr{V},\Z)=\emptyset$);
\item (\cite{Duquesne}) it is proven numerically up to $t<1000$ that the rank of $\mathscr{V}_t$ is $1$.
\end{itemize}

Of course, this example with constant root number on the integer fibers \textbf{does not} satisfy the hypothesis of our theorem - indeed the only finite place $v$ where $\mathscr{V}$ has bad reduction is insipid: it has type $II$ and $P_v(T)=T^2+3T+9$ is such that $\mu_3\subseteq \Q[T]/P_v(T)$ (by Remark \ref{p0II}).

\subsection{Outline of the paper}

In section \ref{sectionformuledusigne}, we decompose the root number into contributions corresponding to the generic places. More precisely, Theorem \ref{formuledusigne} allows to write the root number function $\Z\rightarrow\{\pm1\}$ in the form $$\lambda(M_\E(t))\cdot\prod_{p\mid\delta}{W_p(\E_t)}\prod_{P\text{ not }I_0,I_0^*}{\mathcal{W}_P(u,v)},$$ where \begin{itemize}\item $\delta\in\N$ depends on the coefficients of the Weierstrass equation of $\E$,
\item for an integer $n$, $\lambda(n)=(-1)^{\#\text{ prime divisors of }n}$ is the Liouville function,\item $\mathcal{W}_P$ depends on the type of reduction at an irreducible polynomial $P$ and can be written as $\mathcal{W}_P=g_P\cdot h_P$ for some functions $g_P,h_P$.\end{itemize}
In section \ref{sectionvariation}, we study the variation of the different components in this decomposition. In section \ref{sectionconjectures}, we present the squarefree conjecture and Chowla's conjecture, and we combine those conjectures to create the one variable version of the "Squarefree-Liouville" Sieve (Corollary 4.8) that we use in the proof of Theorem \ref{thmgeneral} in section \ref{sectionproof}. 

Many results of this paper are consequences (or "simply obtained 1-variable versions") of results appearing in \cite{Desjardins1}. Here is a conversion table for the interested reader:

\begin{itemize}
\item Theorem \ref{formuledusigne} comes from \cite[Theorem 3.4.]{Desjardins1}
\item Proposition \ref{existencedunM} comes from \cite[Proposition 4.1.]{Desjardins1}
\item Lemma 3.3 and Lemma 3.4 come respectively from \cite[Lemma 4.2 and Lemma 4.3]{Desjardins1}.
\item Theorem 4.6 comes from \cite[Theorem 2.9]{Desjardins1} and appeared originally in \cite[Th\'eor\`eme 1.3.8]{Desjardinsthese}
\end{itemize}

Lemma 3.2, Proposition 3.6, 3.7 and 3.8 are new results. Observe moreover that, although Theorem \ref{thmgeneral} was proven in the special case where there is a place of multiplicative reduction by Manduchi in \cite[Theorem 2]{Manduchi}, the proof of the general case in this article does not restrict to hers.
\subsection{Acknowledgments} This article was triggered by an email conversation with Fernando Gouv\^ea. I thank him for his interest. I also wish to thank Rena Chu, Marc Hindry and Arul Shankar for helpful discussions and the anonymous referees for several useful suggestions and a very detailled list of typos. 

\section{A formula for the global root number}\label{sectionformuledusigne}

Let $\E$ be the family of elliptic curves over $\Q$ with discriminant $\Delta_\E(T)$ described by the minimal Weierstrass equation
\[Y^2=X^3-27c_4(T)X-54c_6(T).\]

Define the integer: \begin{equation}\label{defdelta}
\delta=2\cdot3\cdot n_4\cdot n_6\cdot n_\Delta \prod_{P,P'\mid\Delta_\E,P\not=P'}{\mathrm{Res}(P, P')}.
\end{equation}
where $P,P'$ run through polynomials associated to generic places of bad reduction and $n_4$, $n_6$ and $n_\Delta$, are the numerators of the contents of the polynomials $c_4(T)$, $c_6(T)$ and $\Delta_\E(T)$.

Moreover, we use from now on the following modified Jacobi symbol:

\begin{definition}\label{modifiedjacobisymbol}
For any integer $n\in\Z$ and prime number $p$, let $n_{(p)}$ denote the integer such that $n=p^{\nu_p(n)}n_{(p)}$.

For each pair of integers $(a,b)\in\Z\times\Z$ and even integer $\delta$,

\begin{equation}\label{quadraticsymboldef}\left(\frac{a}{b}\right)_\delta=\prod_{p\nmid\delta}{\left(\frac{a_{(p)}}{p}\right)^{\nu_p(b)}}\end{equation}
where the product runs through the prime numbers $p\nmid\delta$ and $(\frac{\cdot}{p})$ is the Legendre symbol.
\end{definition}

If $a,b,\delta$ are two-by-two coprime, then the symbol $(\frac{a}{b})_\delta$ is the classical Jacobi symbol. 

We use moreover the Liouville function defined as follows:

\begin{definition}\label{defLiouvillefct} For a non-zero integer $n=\prod_pp^{\nu_p(n)}$, we denote by $\Omega(n)=\sum_p\nu_p(n)$ the number of its prime factors and we define {\em Liouville's function} by the formula
\[\lambda(n)=(-1)^{\Omega(n)}.\]
\end{definition}

\subsection{Decomposition of the root number according to the generic places}

We recall here (the one-variable version of) \cite[Theorem 3.4.]{Desjardins1} that splits the formula of the root number of an integer fiber into contributions of places of bad reduction according to their Kodaira type. Those kind of expression can already be found in Helfgott's preprint \cite{Helfgott}.

\begin{thm} \label{formuledusigne}

Let $\mathscr{E}$ be a family of elliptic curves over $\mathbb{Q}$. Let $\delta$ be defined as in (\ref{defdelta}). We recall that $M_\E$ is the product of the irreducible primitive $P\in\Z[T]$ associated to the places of multiplicative reduction of $\E$.

Then, the root number of $\E_t$, $t\in\Z$, can be written as \[W(\E_t)=\lambda(M_\E(t))\cdot\prod_{p\mid\delta}{W_p(\E_{t})}\cdot\prod_{P\mid\Delta_\E}{h_{P}(t)g_{P}(t)}.\]
The functions $h_{P}$ are given in Table \ref{contributions} and
\[ g_{P}(t)=
\begin{cases}
\left(\frac{\varepsilon_{P}}{P(t)}\right)_{\delta}&\text{ if $\E$ has additive reduction at $P$}\\
\left(\frac{-c_6(t)}{P(t)}\right)_{\delta}\left(\prod_{p\mid\delta}{(-1)^{\nu_p(P(t))}}\right)&\text{ if $\E$ has multiplicative reduction at $P$,}
\end{cases}
\]
where $\left(\frac{\cdot}{\cdot}\right)_\delta$ is the modified Jacobi symbol of Definition \ref{modifiedjacobisymbol} and \[\varepsilon_{P}=\begin{cases}
-1&\text{ if the type is $II$, $II^*$, $I_0^*$ or $I_m^*$ ($m\geq1$)}\\
-2&\text{ if the type is $III$ or $III^*$}\\
-3&\text{ if the type is $IV$ or $IV^*$.}\\
\end{cases}\]
\end{thm}

\begin{center}
\begin{table}
   \begin{tabular}{ | l || c |}
     \hline
Type & $h_{P}(t)$ \\
%& & \\
\hline
$I_0$  & 1 \\
\hline
$I_0^*$ 		& 1 \\
\hline
$II$, $II^*$	& $\prod \limits_{\underset{p^2\mid P(t)}{p\nmid\delta}} 
\begin{cases}
\left(\frac{-3}{p}\right) & \nu_p(P(t))\equiv2,4\mod6 \\
+1 & \text{otherwise.}
\end{cases}$ \\
     \hline
$III$, $III^*$	& $\prod \limits_{\underset{p^2\vert P(t)}{p\nmid\delta}} 
\begin{cases}
\left(\frac{-1}{p}\right) & \nu_p(P(t))\equiv2\mod4\\
+1 & \text{otherwise.}
\end{cases}$ \\
\hline
$IV$, $IV^*$	& $\prod \limits_{\underset{p^2\vert P(t)}{p\nmid\delta}} 
\begin{cases}
\left(\frac{-3}{p}\right) & \nu_p(P(t))\equiv2,3,4\mod6\\
+1 & \text{otherwise.}
\end{cases}$ \\

\hline
$I_m^*$  ($m\geq1$) 	& $\prod \limits_{\underset{p^2\vert P(t)}{p\nmid\delta}} \left(-\left(\frac{-c_6(t)_{(p)}}{p}\right)\right)^{\nu_p(P(t))-1}$ \\
\hline
$I_m$ ($m\geq1$)  	& $\prod \limits_{\underset{p^2\vert P(t)}{p\nmid\delta}} \left(-\left(\frac{-c_6(t)_{(p)}}{p}\right)\right)^{\nu_p(P(t))-1}$ \\ 
 &\\ \hline
\end{tabular}
\caption{\label{contributions}Corrective functions $h_P$ of Theorem \ref{formuledusigne} 
}
\end{table}
 \end{center}

\begin{proof}
The details of the proof can be found in \cite{Desjardins1} where this formula is given for a \emph{rational fiber} $\E_t$, $t\in\Q$. The general idea is the following. The root number can be expressed as
\begin{displaymath}
W(\mathscr{E}_t)=-\prod_{p\mid\delta}{W_p(\E_t)}\prod_{P\mid\Delta_\E}{\mathscr{W}_{P}(t)},
\end{displaymath}
where for each $P$ primitive factor of $\Delta_\E$:
\begin{displaymath}
\W_{P}(t)=\prod_{p \nmid\delta:p\mid P(t)}{W_p(\mathscr{E}_t)}.
\end{displaymath}

Observe that in our case (with $t\in\Z$) there is no contribution to the root number of integer fibers coming from the infinite place $-\deg$.

Now, by the monodromy of the reduction type, each of the factors $W_p(\E_t)$ depends only on the type of $\E_t$ at $P$ and of $n=\nu_p(P(t))$. See Table \ref{tableaumonodromie} for a precise description of the changes in the types. This table can easily be deducted from Tate's algorithm and appeared originally as \cite[Lemma 3.3]{Desjardins1}.

In particular if $n=1$, then the type of reduction of $\E_t$ is the same as the type of $\E$ at $P$. We have thus in this case:
\[W_p(\E_t)=\begin{cases}
\left(\frac{-\varepsilon_{P}}{p}\right)&\text{ if $\E$ has additive reduction at $P$}\\
-\left(\frac{-c_6(t)}{p}\right)&\text{ if $\E$ has multiplicative reduction at $P$.}
\end{cases}\]

\begin{center}
\begin{table}
   \begin{tabular}{| c c l | c c l |}
     \hline
Type of $\E_T$  & $n=\nu_p(P(u,v))$  & Type of $\E_t$ & Type of $\E_T$ & $n=\nu_p(P(u,v))$ & Type of $\E_t$ \\
at $P(T)$& & at $p$& at $P(T)$ && at $p$\\
\hline
     \multirow{2}{*}{$I_m$} & \multirow{2}{*}{$n\geq 1$} & \multirow{2}{*}{$I_{mn}$} & \multirow{2}{*}{$I_m^*$} & $0\mod2$ & $I_{mn}$ \\
&&&& $1\mod2$ &$I_{mn}^*$  \\
\hline
\multirow{6}{*}{$II$} &  $0\mod6$ & $I_0$ &\multirow{6}{*}{$II^*$}& $0\mod6$ & $I_0$\\
& $1\mod6$ & $II$ &&$1\mod6$ &$II^*$\\
&$2\mod6$ & $IV$ && $2\mod6$ &$IV^*$\\
&$3\mod6$ & $I_0^*$& &$3\mod6$&$I_0^*$\\
&$4\mod6$ & $IV^*$ & &$4\mod6$& $IV$\\
&$5\mod6$	& $II^*$ & &$5\mod6$&$II$\\
\hline
 \multirow{4}{*}{$III$}    & $0\mod2$& 	$I_0$ & \multirow{4}{*}{$III^*$}& $0\mod4$ & $I_0$\\
&$1\mod4$&$III$ & &$1\mod4$ & $III^*$ \\
&$2\mod4$& $I_0^*$ && $2\mod4$ & $I_0^*$\\
&$3\mod4$&	$III^*$&& $3\mod4$	& $III$	\\
\hline
 \multirow{3}{*}{$IV$}    & $0\mod3$ & $I_0$ & \multirow{3}{*}{$IV^*$} &$0\mod3$ & $I_0$\\
&$1\mod3$ & $IV$&& $1\mod3$&$IV^*$\\
&$2\mod3$	& $IV^*$	&& $2\mod3$& $IV$\\
\hline
\multirow{2}{*}{$I_0^*$} & $0\mod2$ & $I_0$ & \multirow{2}{*}{$I_0$} & \multirow{2}{*}{$n\geq0$} & \multirow{2}{*}{$I_0$}\\ & $1\mod2$ & $I_0^*$ & 			& & \\
   \hline
\end{tabular}
\caption{\label{tableaumonodromie}Monodromy of the types of reduction 
}
\end{table}
 \end{center}

However, when $n\geq2$, the type of reduction of $\E_T$ is likely to change. For this reason, we introduce a corrective function, denoted by $h_{P}$ and equal to $\frac{g_P}{\W_P}$ so we can write:

\[\W_{P}(t)=h_{P}(t)\begin{cases}
\left(\frac{-\varepsilon_{P}}{P(t)}\right)_\delta&\text{ if $\E_t$ has additive multiplication at $P$}\\
(-1)^{\Omega(P(t)_{(\delta)})}\left(\frac{-c_6(t)}{P(t)}\right)_\delta&\text{ if $\E_t$ has multiplicative multiplication at $P$,}
\end{cases}\]
where $\Omega(n)$ is the number of prime factors of the integer $n$ and $(\frac{\cdot}{\cdot})_\delta$ is the quadratic symbol defined in Definition \ref{modifiedjacobisymbol}.

To complete the proof, we need to prove that the functions in Table \ref{contributions} are indeed the corrective functions $h_P$. This is done case by case according to the Kodaira type of the reduction of $\E_T$ at $P$, using the monodromy given by Table \ref{contributions} \cite[Theorem 3.4]{Desjardins1}.
\end{proof}

\section{Variation of the different components of the root number}\label{sectionvariation}

\subsection{The function $\prod_{p\mid\delta}{W_p(\E_t)}\prod_{P\mid\Delta}{g_{P}}(t)$}\label{constancelocaleg}

As shown in \cite[Prop. 4.1]{Desjardins1}, the "first half" of the formula for the root number described in the previous section respects a certain periodicity: 

\begin{propo}\label{existencedunM}
Let $\E$ be a family of elliptic curves over $\Q$. Let $\delta$ be the integer defined at equation \ref{defdelta}. 

Then, there exist an integer $N_\E\in\N^*$ and a non-zero polynomial $R_\E\in\Z[T]$ such that the function $\varphi_\E:\Z\rightarrow\{-1,+1\}$ defined as \[t\mapsto \varphi_{\E}(t):=\prod_{p\mid\delta}{W_p(\E_{t})}\prod_{P\in\mathscr{B}}{g_{P}(t)}\] has the property that $\varphi_\E(t_1)=\varphi_\E(t_2)$ for every $t_1,t_2\in\Z$ such that
\begin{enumerate}
\item $t_1\equiv t_2\mod N_\E$ and 
\item for every irreducible factor $R_i$ of $R_\E$, one has that $R_i(t_1)$ and $R_i(t_2)$ have the same sign $\epsilon_i$.
\end{enumerate}
\end{propo}

\begin{rem}
Given these integer $N_\E$ and polynomial $R_\E$, assuming (1) and 
replacing Hypothesis (2) by: \begin{enumerate}[(2')] \item for every irreducible factor $R_i$ of $R_\E$, one has that $R_i(t_1)$ and $R_i(t_2)$ have the same sign $\epsilon_i\in\{\pm1\}$ except for a given $i_o$ such that $R_{i_o}(t_1)$ and $R_{i_o}(t_2)$ have opposite sign,\end{enumerate}
then we get that $\varphi_\E(t_1)=-\varphi_\E(t_2)$.

\end{rem}

\begin{proof}
For a detailled proof, we refer to \cite[Prop. 4.1]{Desjardins1}. 
It is based on the following facts:
\begin{enumerate}
\item There exists an integer $\alpha_p$ (that we can suppose minimal) such that the local root numbers at $p\mid\delta$ of the fibers are such that $W_p(\E_{t_1})=W_p(\E_{t_2})$ whenever $t_1,t_2\in\Z$ are such that $t_1\equiv t_2\mod p^{\alpha_p}$.

\item The modified Jacobi symbol respects the property that for any polynomials $g$ and $f$ there exist an integer $N_P$ and a polynomial $R_P$ such that $\left(\frac{f(t)}{g(t)}\right)_\delta$ depends only on the congruence class $t\mod N_P$ and of the sign of the value $R_P(t)$.

We have for $\left(\frac{\epsilon}{P(t)}\right)_\delta$, that $R_P=P$, and if for instance $\epsilon=-1$, then $N_P=4$. 

For general $f,g$ that we suppose such that $\deg g\geq\deg f$ (otherwise, simply exchange $f$ and $g$), we have $R_p=f\cdot g\cdot (g-x^{\deg g-\deg f}f)\cdots$
and we add to this product a finite number of factors $g_i$, $i=1,\cdots$,  of decreasing degree, defined by putting $g_{-1}=f, g_0=(g-x^{\deg g-\deg f}f)$ and iterating the following: suppose $g_i=(g_{i-2}-x^{\deg g_{i-2}-\deg g_{i-1}}g_{i-1})$ is such that $\deg g_i\leq \deg g_{i-1}$ (otherwise interchange their role), put $g_{i+1}=g_{i-1}-x^{\deg g_{i-1}-\deg g_i}g_i$. Repeat until $\deg g_{i+1}=0$.
\end{enumerate}

Therefore, we have that $$N_\E=24\prod_{p\mid\delta}{p^{\alpha_p}}\prod_{P\in\mathscr{M}}{N_P},$$
and $$R_\E=\prod_{P\in\mathscr{M}}{R_P}.$$
\end{proof}

\subsection{The function $h_{P}$}

\begin{lm}\label{constancehpotmult}
Let $\E$ be a family of elliptic curves over $\Q$ with a place of bad reduction associated to the primitive irreducible polynomial $P$.  

Suppose that there exist $t_1,t_2\in\Z$ such that we have
\begin{enumerate}[(1)]
\item\label{hup1} $P(t_1)=c^2l$, where $l$ is squarefree and $\gcd(c,l)=1$,
\item\label{hup2} $P(t_2)=(cq_0)^2\eta$, where $\eta$ is squarefree, $\gcd(c,\eta)=1$, for a certain integer $q_0$,
\end{enumerate}
for some $c\in\Z$.

\begin{enumerate}[A.]\item\label{blop} Suppose that $q_0\not=2,3$ is a prime number such that
\begin{enumerate}[i)]
\item $c_6(t_2)_{(q_0)}\equiv1\mod q_0$ if the reduction of $\E$ at $P$ is additive or potentially multiplicative,
\item $\left(\frac{-3}{q_0}\right)=-1$ if the reduction of $\E$ at $P$ is $II$, $II^*$, $IV$ or $IV^*$, 
\item $\left(\frac{-1}{q_0}\right)=-1$ if the reduction of $\E$ at $P$ is $III$ or $III^*$,
\end{enumerate}
then $h_P(t_1)=-h_{P}(t_2)$;
\item Suppose that $q_0=1$,
then $h_{P}(t_1)=h_{P}(t_2)$.
\end{enumerate}
\end{lm}

\begin{proof}
Let $P$ be a 1-variable irreducible polynomial associated to a place of $\E$ that is neither of type $I_0$ nor of type $I_0^*$.
The function $h_P$ is a product over the square factors of $P(t)$ not dividing $\delta$. By hypothesis \ref{hup1} and \ref{hup2}, the square factors of $P(t_1)$ and $P(t_2)$ are the same except for $q_0$. Hence,
$$h_P(t_1)=%\left(\frac{\lambda_P}{p}\right)
\lambda_P h_P(t_2),$$
where $\lambda_P$ depends on the type of reduction of $\E$ at $P$.

\begin{itemize}
\item[$\bullet$] Suppose the reduction at $P$ has type $II$, $II^*$, $IV$ or $IV^*$. Then $\lambda_P=\left(\frac{-3}{p}\right)$ and by hypothesis, it is equal to $-1$. Hence $h_P(t_1)=-h_P(t_2)$.

\item[$\bullet$] Suppose the reduction at $P$ has type $III$, $III^*$. Then $\lambda_P=\left(\frac{-1}{p}\right)$ and by hypothesis, it is equal to $-1$. Hence $h_P(t_1)=-h_P(t_2)$.

\item[$\bullet$] Suppose the reduction at $P$ has type $I_m^*$  or $I_m$ ($m\geq1$). Then $\lambda_P=\left(-\left(\frac{c_6(t)_{(p)}}{p}\right)\right)$.

We put $\alpha:=P(t_1)=c^2l$ and $\beta:=P(t_2)=c^2q_0^2\eta$. Then by Theorem \ref{formuledusigne}, one has

\[h_{P}(t_2)=\prod_{p\nmid\delta;p\mid cq_0}{\begin{cases}
-\Big(\frac{-c_6(t_2)_{(p)}}{p}\Big) & \text{if }2\nu_p(q_0c)\equiv2,4\mod6\\
+1 & \text{otherwise.}
\end{cases}}\]
\[=-\Big(\frac{-c_6(t_1)_{(q_0)}}{q_0}\Big)\prod_{p\nmid\delta;p\mid c}{\begin{cases}
-\Big(\frac{-c_6(t_1)_{(p)}}{p}\Big) & \text{if }\nu_p(c)\text{ is even,}\\
+1 & \text{otherwise.}
\end{cases}}\]
\[=-\Big(\frac{-c_6(t_1)_{(q_0)}}{q_0}\Big)\cdot h_{P}(t_1).\]

By assumption on $q_0$, one has $q_0^{-2}\frac{-c_6(t_1)}{P(t_1)^2}\equiv1\mod q_0$. If we put $P(t_1)=q_0^2\mu$ where $\mu$ is an integer coprime to $q_0$, one has $q_0^{-6}(-c_6(t_1))\equiv \mu^2\mod q_0$. Thus, one has $\left(\frac{(-c_6(t_1))_{(q_0)}}{q_0}\right)=+1$.

Thus we have the equality
\[h_{P}(t_1)=-h_{P}(t_2).\]
\item If $q_0=1$, then the products in $h_P(t_1)$ and $h_P(t_2)$ have exactly the same factors. Thus $h_P(t_1)=h_P(t_2)$.
\end{itemize}
\end{proof}

Be aware however, that in certain cases a number $q_0$ with properties \ref{blop} ii) or iii) does not exist - this is the subject of the next subsection.

\subsubsection{Existence of $q_0$}

If the reduction at a polynomial $P$ is such that 
\begin{itemize}
\item[$\bullet$]{\bf $P$ has type $I_0^*$}
then the theorem does not apply. Indeed, $h_P(t)=+1$ for all $t\in\Z$

\item[$\bullet$]{\bf $P$ has type $II$, $II^*$, $IV$, $IV^*$} then we use the following Lemma.

\begin{lm}\label{constance24}
Let $P\in\Z[T]$ be a polynomial associated to a place of type $II$, $II^*$ $IV$ or $IV^*$. We assume that \begin{equation}\label{proprieteII}\mu_3\subseteq \Q[T]/P(T),\end{equation} where $\mu_3$ is the group of third roots of unity.

Then for all $t\in\Z$ one has $h_{P}(t)=+1$.
\end{lm}

\begin{proof}
Let $P(T)\in\Z[T]$ be an irreducible polynomial and let $K=\Q[T]/P(T)$. Suppose that $\Q(\mu_3)\subseteq K$. In \cite[Proposition 2.2]{JulieBartosz}, it is proven that there exist $A,B\in\Z[T]$ such that $P=A^2+3B^2$ and this implies that for every $t\in\Z$, if $p\mid P(t)$ for some prime number, then $p\equiv1\mod 3$. This means that any $p$ dividing a value of $P$ is such that $\left(\frac{-3}{p}\right)=+1$. We have:
\begin{align*}h_P(t)&=\prod_{p\nmid\delta\text{, }p^2\mid P(t)}{\begin{cases}\left(\frac{-3}{p}\right)&\text{if }v_p(P(t))\equiv2,4\mod6\\+1&\text{otherwise}\end{cases}}\\
&=+1.
\end{align*}
\end{proof}

\begin{rem}\label{p0II}
One can easily give examples of polynomials $P\in\Z[T]$ satisfying the hypothesis $\mu_3\subseteq\Q[T]/P_i(T)$ for every irreducible factors $P_i$ of $P$. We have in particular those of the form \[P(T)=3A(T)^2+B(T)^2,\]
where $A(T),B(T)\in\Z[T]^*$ are coprime.
%Although t
This polynomial is not necessarily irreducible in general, but for 
every irreducible factor $P_i$ and the corresponding field $K_i=\Q[T]/(P_i(T))=\Q(\alpha_i)$ where $\alpha_i$ is a root of $P_i$, one has $3A(\alpha_i)^2+B(\alpha_i)^2=0$, thus $-3=(B(\alpha_i)A(\alpha_i)^{-1})^2$ and this proves that $\mu_3\subset K_i$.

\end{rem}

If a polynomial $P$ does not satisfy the property (\ref{proprieteII}), then there exist a prime number $p_o$ and an integer $t$ such that $p_o^2\mid P(t)$ and $\left(\frac{-3}{p_o}\right)=-1$.

\item[$\bullet$]{\bf $P$ has type $III$, $III^*$}

\begin{lm}\label{constance3}
Let $P\in\Z[T]$ be an irreducible polynomial associated to a place of type $II$, $II^*$, $IV$ or $IV^*$. Assume that \begin{equation}\label{proprieteIII}\mu_4\subseteq \Q[T]/P(T),\end{equation} where $\mu_4$ is the group of the fourth roots of unity.

Then for all $t\in\Z$, one has $h_{P}(t)=+1$.
\end{lm}

\begin{proof}
The proof is similar to Lemma \ref{constance24}.
\end{proof}

\begin{rem}\label{p0III}
One can easily give examples of polynomials satisfying the hypothesis $\mu_4\subseteq\Q[T]/(P_i(T))$ for every factor $P_i$; in particular those of the form \[P(T)=A(T)^2+B(T)^2,\]
where $A(T),B(T)\in\Z[T]^*$ are coprime.
\end{rem}

If a polynomial $P$ does not satisfy the property (\ref{proprieteIII}), then there exist a prime number $p_o$ and an integer $t$ such that $p_o^2\mid P(t)$ and $\left(\frac{-1}{p_o}\right)=-1$.

\item[$\bullet$]{\bf The reduction at $P$ is multiplicative or potentially multiplicative}

Now, we present a general result on values of polynomials which will allow us to prove the existence of $t_1$, $t_2$ and $q_0$ for which the hypothesis \ref{blop} i) of Lemma \ref{constancehpotmult} is satisfied.

\begin{lm} \label{variationpotmult} \cite[Lemma 2.3]{Manduchi} Let $Q(T)$ and $P(T)\in\Z[T]$ be such that $Q(T)$ is non-constant. Let $\mathrm{Res}(P,Q)$ be the resultant of $P$ and of $Q$, and let $\Delta_Q$, be the discriminant of $Q$. Suppose $\mathrm{Res}(P,Q)$ and $\Delta_Q$ are non-zero. Let $\mathscr{P}_0$ be a finite set of prime numbers.

Then there exist a prime number $p_0\not\in\mathscr{P}_0$ and $n$ a positive integer such that $p_0^2\mid Q(n)$ and $p_0^{-2}P(n)Q(n)\equiv1\mod p_0$.
In particular, $p_0^2\mid\mid Q(n)$ and $p_0\nmid P(n)$.
\end{lm}

We refer to \cite{Manduchi} for a proof of this elementary lemma.
Observe that there is no need of the Squarefree conjecture in Lemma \ref{variationpotmult}.

Thus, for any family of elliptic curves $\E$ with a place of type $I_m$ or $I_m^*$ ($m\geq1$) whose associated polynomial is $Q$, put $P=-\frac{c_6(t)}{Q(t)^3}$. Manduchi's lemma on $P$ and $Q$ guaranties the existence of a $p_0$ such that $p_o^{-2}P(t)Q(t)\equiv c_6(t)\cdot\square\mod p_0\equiv1\mod p_0$, and thus $\left(\frac{-c_6(t)}{p_0}\right)=+1$.

\end{itemize}

\subsection{Variation of the global root number}

\begin{propo}\label{variation3}
Let $\E$ be a family of elliptic curves% with no place $I_m$% and satisfying the hypothesis of Theorem \ref{thmgeneral} or \ref{thmtri}
. Let $N_\E$ be the integer and $R_\E(T)$ be the polynomial given by Proposition \ref{existencedunM}. 
Let $t_1,t_2\in\Z$, and suppose they satisfy the following properties. 

\begin{enumerate}
\item We have $t_1\equiv t_2\mod N_\E$, a non-zero congruence class.
\item We have $\lambda(M(t_1))=\lambda(M(t_2))$
\item For every primitive factor $R_i$ of $R_\E$ and $j=1,2$ one has $R_i(t_j)>0$. 

\item For a certain $Q_0$ and some $c\in\Z$, one has
\begin{enumerate}
\item $Q_0(t_1)=c^2l$ where $l$ is a squarefree integer coprime to $N_\E$,
\item $Q_0(t_2)=c^2q_0^2l'$ where $l'$ is a squarefree integer coprime to $N_\E$, and $q_0$ is a prime number which does not divide $\delta$ and such that 
\begin{enumerate}\item $-p_0^{-6}c_6(t_i)$ is a square modulo $q_0$ for $i=1,2$, if $Q_0$ has type $I_m^*$ ($m\geq1$),
\item $\left(\frac{-3}{p_0}\right)=-1$ if $Q_0$ has type $II$, $II^*$, $IV$ or $IV^*$;
\item $\left(\frac{-1}{p_0}\right)=-1$ if $Q_0$ has type $III$, $III^*$.
\end{enumerate}
\end{enumerate}
\item For every $Q\not=Q_0$ of bad, non $I_0^*$ reduction,
one has $Q(t_1)=c_Q^2l_Q$ and  $Q(t_2)=c_Q^2l_Q'$where $l_Q$ and $l_Q'$is squarefree integers coprime to $N_\E$, and $c_Q\in\Z$.
%\item  where $l_Q'$ is a squarefree integer coprime to $N_\E$.
%\end{enumerate}
\end{enumerate}

Then, we have \[W(\E_{t_1})=-W(\E_{t_2}).\]
\end{propo}

\begin{proof}
By Theorem \ref{formuledusigne}, we have:
$$W(\E_{t_1})=\lambda(M_\E(t_1))\prod_{p\mid\delta}{W_p(\E_{t_1})}\prod_{P\in\mathscr{M}}{g_P(t_1)}\prod_{P\in\mathscr{B}}{h_P(t_1)}.$$ 
By Hypotheses (1), (2) and (3), we have
$$W(\E_{t_1})=\lambda(M_\E(t_2))\prod_{p\mid\delta}{W_p(\E_{t_2})}\prod_{P\in\mathscr{M}}{g_P(t_2)}\prod_{P\in\mathscr{B}}{h_P(t_1)}.$$ 

The rest of the proof is based essentially on Propositions \ref{variation3} or \ref{variation3bis}. By hypotheses (4), 
$h_{Q_0}(t_1)=-h_{Q_0}(t_2)$
and finally by Hypotheses (5), for every $Q\not=Q_0$ of bad reduction
$h_Q(t_1)=h_Q(t_2)$. Hence

\begin{align*}W(\E_{t_1})&=-\lambda(M_\E(t_2))\prod_{p\mid\delta}{W_p(\E_{t_2})}\prod_{P\in\mathscr{M}}{g_P(t_2)}\prod_{P\in\mathscr{B}}{h_P(t_2)}\\
&=-W(\E_{t_2})\end{align*}

\end{proof}

Very similarly we have:
\begin{propo}\label{variation3bis}
Do the same hypotheses (1) to (3) as in Proposition \ref{variation3}, and suppose additionally that for a certain $Q_0$ and some $c\in\Z$, one have $Q_0(t_1)=c^2l$ and $Q_0(t_2)=c^2l'^2$ where $l,l'$ are squarefree integers coprime to $N_\E$. Then $$W(\E_{t_1})=W(\E_{t_2}).$$
\end{propo}

There is also another approach to making the root number vary:

\begin{propo}\label{variationimpaire}

Let $\E$ be a family of elliptic curves over $\Q$ with no place $I_m$ nor $I_m^*$. Let $N_\E$ be the integer and $R_\E(T)$ be the polynomial given by Proposition \ref{existencedunM}.
Let $t_1,t_2\in\Z$. Suppose they satisfy the following properties. 

\begin{enumerate}
\item We have $t_1\equiv t_2\mod N_\E$, a non-zero congruence class.
\item For a fixed primitive factor $R_{i_o}$ of $R_\E$ one has $R_i(t_1)>0$ and $R_i(t_2)<0$
\item For every other primitive factor $R_i$ of $R_\E$ one has $R_i(t_1)>0$ and $R_i(t_2)>0$. 

\item For every place of bad (but not insipid) reduction, one has
%\begin{enumerate}\item 
$Q(t_1)=c_Q^2l_Q$ and $Q(t_2)=c_Q^2l_Q'$ where $l_Q$ and $l_Q'$ are squarefree integers coprime to $N_\E$.
%\item  where $l_Q'$ is a squarefree integer coprime to $N_\E$.
%\end{enumerate}
\end{enumerate}

Then, we have \[W(\E_{t_1})=-W(\E_{t_2}).\]
\end{propo}

\begin{proof} The crux of the proof is the function $\varphi_\E:\Z\rightarrow\{-1,+1\}$ defined in Proposition \ref{existencedunM} as \[t\mapsto \varphi_{\E}(t):=\prod_{p\mid\delta}{W_p(\E_{t})}\prod_{P\in\mathscr{B}}{g_{P}(t)}.\] Let $R_\E=\prod_{i}{R_i}$ be the polynomial given by Proposition \ref{existencedunM}. For a fixed primitive factor $R_{i_o}$ of $R_\E$ one has $R_i(t_1)>0$ and $R_i(t_2)<0$. For every other primitive factor $R_i$ of $R_\E$ one has $R_i(t_1)>0$ and $R_i(t_2)>0$. Then by construction of $R$, we have $\varphi_\E(t_1)=-\varphi_\E(t_2)$.

The other parts of the root number formula do not vary by Proposition \ref{variation3bis}. Hence $W(\E_{t_1})=-W(\E_{t_2})$.

\end{proof}

\section{Analytical conjectures and sieves}\label{sectionconjectures}

In this section, we recall the two analytical number theory conjectures applied on a polynomial $f(T)\in\Z[T]$ that appear in Theorem \ref{thmgeneral}:
\begin{enumerate}
\item Chowla's conjecture,
\item Squarefree conjecture.
\end{enumerate}
They respectively correspond to asymptotic approximations of the proportion of the values of $f(t)$ with the following properties:
\begin{enumerate}
\item a certain parity of the number of prime factors,
\item a fixed square part.
\end{enumerate}

Notice that we recall here exclusively the statements in the case where $f(t)$ is one-variable polynomial since it is sufficient for our purpose, and that for details of the 2-variable homogenous case, we refer the interested reader to our previous paper \cite[Section 2]{Desjardins1}.

We will make a certain number of assumptions on the polynomial $f(T)$ throughout this section:
first, it is natural to assume that $f$ is {\it primitive} (i.e. that its content is equal to 1), and second, that $f$ is squarefree (i.e. that there is no polynomial $f_0$ such that $f_0^2\mid f$), which is the same as supposing that its discriminant $D_f$ is non-zero.
  
We denote by $|\cdot|$ the usual absolute value on $\bR$ or the max norm on $\bR^2$. Let $\cA=a+N\Z$ (where $N\not=0$) be an arithmetic progression. We also introduce the notations
\[\bZ(X)=\left\{t\in\bZ\;|\;| t |\leq X\right\},\;\; \cA(X):=\left\{t\in\cA\;|\;|t|\leq X\right\}\quad{\rm and}\quad \A(X):=\sharp\cA(X),\]
so that $\A(X)$ is roughly proportional to $X$. 
More precisely, we have $\A(X)\sim \frac{X}{N}$.

\subsection{Squarefree conjecture}

This conjecture estimates the proportion of squarefree values of a polynomial in an arithmetic progression. 

\noindent{\bf Notations.} Put 
 $\delta_{f}:=\gcd\left\{f(t)\;|\;t\in\bZ\right\}$, then  let $d_{f}$ be the smallest integer such that $\delta_f/d_{f}$ is squarefree. Write
 $d_{f}=\prod_pp^{\nu_p}$.
We denote by $t_f({p})$ the number of solutions modulo $p^{2+\nu_p}$ of $f(t)d_{f}^{-1}\equiv 0\mod p^2$, or $f(t)\equiv 0\mod p^{2+\nu_p}$ and we put
\[C_f:=\prod_p\left(1-\frac{t_f({p})}{p^{2+\nu_p}}\right).\]
Note that %if the polynomial is primitive, we have $\tilde{t}_f({p})=t_f({p})$ for $p>d:=\deg f$. Moreover, 
the product defining $C_f$ is absolutely convergent since $t_f(p)=O(1)$.

To study the ``almost"\footnote{We allow the squarefree part of the values of $f$ to be a certain fixed integer.} squarefree values we proceed as follows. Let $\phi:\Z\rightarrow\cA$ the bijection with $\phi(t)=a+Nt$. We put
$g:=f\circ\phi$, $d_{f,\cA}:=d_{g}=\prod_{p}{p^{\mu_p}}$, and define $t_{f,\mathcal{A}}(p)$ as the number of solutions modulo $p^{2+\nu_p}$ of $g(t)d_{f,\mathcal{A}}^{-1}\equiv0\pmod{p^{2+\mu_p}}$,
\[C_{f,\cA}:=\frac{1}{N}\prod_{p\nmid N}{\left(1-\frac{t_{f,\mathcal{A}}(p)}{p^{2+\mu_p}}\right)}\] 
and let $Sqf_\cA(X)$ to be the number of elements $t\in\cA(X)$ such that $f(t)$ is squarefree.
We can write the Squarefree conjecture in the following form:

\begin{conjecture}\label{sfc1} (Squarefree Conjecture on arithmetic progressions)
\begin{equation}\label{nbdesqfA}Sqf_\cA(X)=C_{f,\cA}X+o(X)%\left(\frac{X^2}{\log (X)^{1/2}}\right)
\end{equation}
\end{conjecture}

It is known to hold in the following cases:

\begin{thm}\label{theosqf}  (Hooley \cite{Hool}) Let $f\in\Z[T]$ be a primitive squarefree polynomial with integer coefficients in one variable. Suppose that every irreducible factor of $f$ has degree at most $3$. 
Then $f$ verifies the Squarefree conjecture on arithmetic progressions.
\end{thm}

It leads to the following \emph{squarefree sieve}\footnote{A homogeneous squarefree sieve is given in \cite{GM} and is unconditional for 2-variable homogeneous $f\in\Z[X,Y]$ whose irreducible factor have all degree less than 6.}:

\begin{coro}\label{sfcrible}
Let $f(T)\in\Z[T]$ be a polynomial such that no square of a non-unit in $\Z[T]$ divides $f(T)$, and such that every factor of $f$ has degree $\leq3$.

Fix 
\begin{itemize}
\item a sequence $S=(p_1,\cdots,p_s)$ of distinct prime numbers and
\item a sequence $T=(t_1,\cdots,t_s)$ of non-negative integers.
\end{itemize}
Let $N$ be an integer such that $p_1^{t_1+1}\cdots p_s^{t_s+1}\mid N$ and that $p^2\mid N$ for all primes $p\leq(\deg f)$.

Suppose that there exists an integer $a$ such that 
\begin{enumerate}
\item $f(a)\not\equiv0\mod p^2$, whenever $p\mid N$ and $p\not=p_i$ for any $i$,
\item and such that $v_{p_i}(f(a))=t_i$ for every $i=1,\dots,s$.
\end{enumerate}

Then there are infinitely many integers $t$ such that 
\begin{enumerate}
\item $t\equiv a\mod N$,
\item and such that $f(t)=p_1^{t_1}\cdots p_s^{t_s}\cdot l$, where $l$ is squarefree and $v_{p_i}(l)=0$ for all $i=1,\dots,s$.
\end{enumerate}
\end{coro}

\begin{rem}
Let $R\in\Z$ be an integer, than we can as well require that the infinitely many integers $t$ given by the sieve are such that $t\geq R$ (or equivalently such that $t\leq R$).
\end{rem}

\medskip

\subsection{Chowla's conjecture}

This second conjecture estimates the proportion of the values $f(t)$ with a certain parity of the number of prime factors. Recall that Liouville's function (Definition \ref{defLiouvillefct}) already appears since section \ref{sectionformuledusigne}.

\begin{conjecture}\label{CC} (Chowla's conjecture) Let $f$ be a primitive squarefree polynomial with integer coefficients. The following estimation holds for every arithmetic progression $\cA$:
\begin{equation}
\sum_{t\in\cA(X)}\lambda(f(t))=o(X)\end{equation}
\end{conjecture}

\medskip

A 2-variable homogeneous version is known to hold for polynomials up to degree 3. This is proven for non-irreducible polynomials in \cite{HelfChowla2}, and for irreductible polynomials in \cite{HelfChowla} and more recently in \cite{Lachand}. In our 1-variable case however Conjecture \ref{CC} is known to hold in the following case:

\begin{thm}[Hadamard -- de la Vall\'ee Poussin]\label{theochowla}  Let $f\in\Z[T]$ be a squarefree polynomial. Chowla's conjecture holds if $f$ is linear.
\end{thm}

\subsection{Combination of the conjectures}

In \cite[Theorem 2.9]{Desjardins1}, the author proves an analytical theorem for homogeneous polynomials of two variables indicating a sort of independence between the two properties of the values of a polynomial
\begin{itemize}
\item being ``squarefree"
\item having the same ``parity of the number of factors".
\end{itemize}
In this paper we need the following 1-variable analogue of that result, whose proof can be found in \cite[Th\'eor\`eme 1.3.8]{Desjardinsthese}:

\begin{thm}\label{conjecturesfacteurs}\cite[Theorem 2.9]{Desjardins1} Fix an element $\epsilon\in\{1,-1\}$. Let $f,g,h\in\Z[T]$ be squarefree primitive polynomials with no common factor. Assume the Squarefree conjecture holds for $f$ and $g$ and that Chowla's conjecture holds for $f$. Then for any arithmetic progression $\cA=N\Z+a$ (where $a$, $N\not=0$  are integers such that $a$ and $N$ are coprime), define $T(X)$ to be the number of pairs of integers $t\in\cA(X)$ such that
\begin{enumerate}
\item $\frac{f(t)g(t)}{d_{fg,\cA}}$ is not divisible by $p^2$ for any prime $p$ such that $p\nmid N$;
\item $\lambda(f(t))=\epsilon$
\item for all factor $h_i$ of $h$, one has $h_i(t)>0$.
\end{enumerate}

Then
\begin{equation}
T(X)=\frac{C_{fg,\cA}}{2}X+o\left(X\right).
\end{equation}
\end{thm}

Combining Theorem \ref{conjecturesfacteurs} with Theorem \ref{theosqf} and Theorem \ref{theochowla}, we obtain:
\begin{coro}
Fix $\epsilon\in\{1,-1\}$. Let $f,g,h\in\Z[T]$ be squarefree polynomials. Assume that every factor of $g$ has degree at most 3 and that $\deg f=1$. For an arithmetic progression $\cA$, let $T(X)$ be the counting function as defined in Theorem \ref{conjecturesfacteurs}. Then the following estimate holds
\begin{equation}
T(X)=\frac{C_{fg,\cA}}{2}X+o(X).
\end{equation}
\end{coro}

\subsubsection{A one variable version of the "Squarefree-Liouville" Sieve}

\begin{coro}\label{sflcrible}
Let $f(T),g(T)\in\Z[T]$ be polynomials. Assume that they are coprime, that no square of a nonunit in $\Z[T]$ divides $f(T)$ or $g(T)$, that every irreducible factor of $g$ has degree $\leq3$, and that $\deg f\leq1$.

Let $R(T)\in\Z[T]$ be a polynomial and $R= \gamma\cdot\prod_{0\leq i\leq r}{R_i}$ its decomposition in primitive factors. Suppose $sgn(\gamma)=+1$.
Fix
\begin{itemize}
\item a sequence $S=(p_1,\dots,p_s)$ of distinct prime numbers and
\item a sequence $T=(t_1,\dots,t_{s},t'_1,\dots,t'_{s})$ of nonnegative integers.
\end{itemize}
Let $N$ be an integer such that $p_1^{t_1+t'_1+1}\dots p_s^{t_s+t'_s+1}\mid N$ and that $p^2\mid N$ for all primes $p<(\deg f+\deg g)$. 

Suppose that there exists an integer $a$ such that 
\begin{enumerate}
\item\label{paszero} $f(a)g(a)\not\equiv0\mod p^2\text{, whenever }p\mid N \text{ and }p\not=p_i\text{ for any }i,$
\item\label{memesecteur} for every $1\leq i\leq r$, one has $R_i(t)> 0$,
\item\label{puissanceok}
$v_{p_i}(f(a))=t_i$ and $v_{p_i}(g(a))=t'_i\text{ for every }i=1,\dots,s$
\item and
$\lambda(f(a))=\epsilon.$
\end{enumerate}
Then there are infinitely many integers $t$ such that
\begin{enumerate}[1.]
\item $u\equiv a\mod N,$
\item for every $i$, one has $R_i(t)> 0$,
\item 
$\lambda(f(t))=\epsilon,$
\item and 
$f(t)=p_1^{t_1}\dots p_s^{t_s}\cdot l,$ and $g(t)=p_1^{t'_1}\dots p_s^{t'_s}l'$
where $l$ and $l'$ are squarefree and $v_{p_i}(l)=v_{p_i}(l')=0$ for all $i=1,\cdots,s$.
\end{enumerate}

\end{coro}
\section{Proof of the theorem \ref{thmgeneral}}\label{sectionproof}

\begin{thm}\label{thmgeneral1}(=Theorem \ref{thmgeneral})
Let $\E$ be a family of elliptic curves over $\Q$. Suppose that there exists at least one polynomial $P_o(T)\in\Q[T]$ such that the reduction of $\E$ at $P_o(T)$ is either
\begin{enumerate}
\item\label{hyp1} multiplicative;
\item\label{hyp2} additive potentially multiplicative;
\item\label{hyp3} semi-stable and $\deg P_o$ is odd;
\item\label{hyp4} or additive potentially good and if the Kodaira symbol of the reduction is 
	\begin{enumerate}\item $II$, $II^*$, $IV$ or $IV^*$, then $$\mu_3\not\subseteq \Q[T]/P_o(T);$$
	\item $III$ or $III^*$, then $$\mu_3\not\subseteq \Q[T]/P_o(T).$$
	\end{enumerate}
\end{enumerate}

Suppose moreover that 
\begin{enumerate}[(i)]
\item $M_\E=\prod_{P\in\mathscr{M}}{P}$ respects Chowla's conjecture; 
\item $B_\E=\prod_{P\in\mathscr{B}}{P}$ respects the squarefree conjecture.
\end{enumerate}
Then the sets $W_+(\E,\Z)$ and $W_-(\E,\Z)$ are both infinite.
\end{thm}

The general strategy is to construct infinitely many pairs of integers $(t_1,t_2)$ such that $W(\E_{t_1})=-W(\E_{t_2})$. 

Recall that a place (or an irreducible polynomial associated to this place) is \emph{insipid} if it is of type \begin{itemize}
\item $I_0^*$
\item $II$, $II^*$, $IV$ or $IV^*$, and $$\mu_3\subseteq \Q[T]/P(T);$$
\item or $III$ or $III^*$, and $$\mu_4\subseteq \Q[T]/P(T).$$
\end{itemize} 

\begin{proof} Let $\frac{n_4}{d_4}$, $\frac{n_6}{d_6}$, $\frac{n_\Delta}{d_\Delta}$ be the respective coprime numerators and denominators of the contents of the polynomials $c_4(T)$, $c_6(T)$, $\Delta(T)$ associated to $\E$. (Observe that the only factors of $d_4,d_6,d_\Delta$ are $2$ or $3$.) Put as in Section \ref{sectionformuledusigne}
\[\delta=2\cdot3\cdot n_4\cdot n_6\cdot n_\delta\prod_{Q,Q'\in\mathscr{B}}{\mathrm{Res}(Q,Q'}).\]

Let $N=N_\E$ be the integer and $R_\E$ be the polynomial given by Proposition \ref{existencedunM} (choose the minimal such integer $N_\E$ and polynomial $R_\E$ of lowest degree). Write $R_\E=R=\gamma R_1\cdots R_r$, the decomposition in primitive polynomial. Let us choose $R$ such that $sign(\gamma)=+1$.
By definition,
the function $t\mapsto\prod_{p\mid\delta}{W_p(\E_{t})}\prod_{P\in\mathscr{B}}{g_{P}(t)}$ is constant when $t$ stays in a congruence class modulo $N$ and when $t$ is in a connected component of \[\R-\{t\in\R \ \vert R(t)=0\}.\]

For each $p\mid N$, put $\alpha_p=\nu_p(N)$. Moreover, write $N=2^{\alpha_2}3^{\alpha_3}p_1^{\alpha_{p_1}}\cdots p_s^{\alpha_{p_s}}$ the factorisation into distinct prime numbers.

Put \[S=(2,3,p_1,\dots,p_s),\] 
 and \[T=(0,\dots,0).\]

Let $\mathfrak{a}_2\mod 2^{\alpha_2}$ be a congruence class such that
for all $P_i$ of bad reduction (except the insipid ones):
\begin{equation*}\label{classe2}
P_i(\mathfrak{a}_2)\not\equiv0\mod2^{\alpha_2}.
\end{equation*}

Let $\mathfrak{a}_3\mod 3^{\alpha_3}$ be a congruence class such that for all $P_i$ of bad reduction (except the insipid ones):
 \begin{equation*}\label{classe3}
P_i(\mathfrak{a}_3)\not\equiv0\mod3^{\alpha_3}.
\end{equation*}

Let also, for each $p\mid N$ such that $p\not=2,3$, be classes $\mathfrak{a}_p\mod p^{\alpha_p}$ such that for all $P_i$ of bad reduction (except the insipid ones) we have:
\[P_i(\mathfrak{a}_p)\not\equiv0\mod p^{\alpha_p}.\]
As by assumption the polynomial $P_i$ have content 1, such classes $\mathfrak{a}_p$ exist for every $p\mid N$.

By the Chinese Remainder Theorem, there exist integers $a$ satisfying
\begin{equation}\label{eq:chinois1}
a\equiv\begin{cases}
\mathfrak{a}_2\mod 2^{\alpha_2}, & \\
\mathfrak{a}_3\mod 3^{\alpha_3}, & \\
\mathfrak{a}_p\mod p^{\alpha_p} & \text{for all $p\mid N$.} \\
\end{cases}
\end{equation}

Now we use the hypotheses: there exists an irreducible polynomial such that the reduction respects one of the properties from \ref{hyp1} to \ref{hyp4}. %We can either choose $P_o$ to be non-insipid, or of type $I_0^*$ and such that $\deg P_o$ is odd.

\begin{itemize}
\item Suppose that we can choose $P_o$ to have multiplicative reduction (property \ref{hyp1}).
Then 
\begin{enumerate}[i.]
\item choose $\mathfrak{a}_p$ such that $R_i(t)>0$ for all $1\leq i\leq r$.
\item put $f:=M_\E$.
\end{enumerate}
\item Suppose that there's no multiplicative place and that $P_o$ is not insipid (property \ref{hyp2} or \ref{hyp4}), then
\begin{enumerate}[i.]
\item choose $a$ (as in (\ref{eq:chinois1})) such that $R_i(t)>0$ for all $1\leq i\leq r$.
\item put $f:=1$.
\end{enumerate}
\item Suppose that $P$ of reduction type $I_0^*$ with $\deg P$ odd (property \ref{hyp3}),
then
\begin{enumerate}[i.]
\item choose $a$ (as in (\ref{eq:chinois1})) such that $R_i(t)>0$ for all $1\leq i\leq r$ except for $R_i=P_o$ where $R_{i_0}<0$.
\item put $f:=1$.
\end{enumerate}
\end{itemize}

In any case, put $g=\prod{P}$, where $P$ runs through the non-insipid polynomials.

If $f\not=1$ (property \ref{hyp1}: there exists a finite place of multiplicative reduction), use the Liouville-Squarefree Sieve given by Corollary \ref{sflcrible} on $f,g,S,T,N$ and $a$ as previously, to prove that there exists a set $\mathscr{F}_1$ of infinitely many $t\in\Z$ such that 
\[g(t)=l\text{, }\quad \text{ and }\lambda(f(t))=+1\]
where $l$ is a squarefree integer coprime to each $p\in S$ by our choice of $S$ and $T$.
This set $\mathscr{F}_1$ may be chosen such that  $\forall t\in\mathscr{F}_1$, $R_i(t)>0$ for all $1\leq i\leq r$.%\footnote{See the argument at the end of Remark \ref{thmsurlessecteurs}.}

If $f=1$, use similarly the Squarefree Sieve given by Corollary \ref{sfcrible} on $g,S,T,N$ and $a$, to obtain the set $\mathscr{F}_1$. As previously, the set $\mathscr{F}_1$ may be chosen such that  $\forall t\in\mathscr{F}_1$, $R_i(t)>0$ for all $1\leq i\leq r$.

By Proposition \ref{variation3} part B, we have that for all $t,t'\in\mathscr{F}_1$,
\[W(\E_{t})=W(\E_{t'}).\]
 
Suppose that the $P_o$ that we fixed has type $I_0^*$ and has odd degree (property \ref{hyp3}). Then we similarly define $\mathscr{F}_2$ such that for all $P\not=P_o$ one has $R_P(t)>0$, and $R_{P_0}(t)<0$. By Proposition \ref{variation3}, all the elements of $\mathscr{F}_2$ are such that their fibers on $\E$ have the same root number. Then by Proposition \ref{variationimpaire} every pair $(t_1,t_2)$ with $t_1\in\mathscr{F}_1$, and $t_2\in\mathscr{F}_2$, is such that \[W\left(\E_{t_1}\right)=-W\left(\E_{t_2}\right).\]

For the other cases, according to the type of $P_o$, choose a prime number $q_0$ with one of the following property:
\begin{enumerate}
\item Suppose that $P_o$ has type $I_m^*$ or $I_m$ (property \ref{hyp1} or \ref{hyp2}).
Put $P(T)=-c_6(T)/Q(T)^3$.
By Lemma \ref{variationpotmult} applied to $P(T)$, $Q(T)$ and $S=:\mathscr{P}_0$, there exist $q_0\not\in S$ and $m_0\leq0$ an integer such that $q_0^2\mid Q(m_0)$ and that $-q_0^{-2}P(m_0)Q(m_0)$ is a square modulo $q_0$. 
\item Suppose that $P_o$ has type $II$, $II^*$, $IV$, $IV^*$ (property \ref{hyp4} (a)), then let $q_o$ be such that $\left(\frac{-3}{q_0}\right)=-1$. This prime exists whenever $P_o$ is not insipid by Remark \ref{p0II}
\item Suppose that $P_o$ has type $III$, $III^*$ (property \ref{hyp4} (b)), then let $q_0$ be such that $\left(\frac{-1}{q_0}\right)=-1$. This prime exists whenever $P_o$ is not insipid by Remark \ref{p0III}.
\end{enumerate}

Consider the sequences
\[S'=(2,3,p_1,\dots, p_s,q_0)\]
and
\[T'=(0,\dots,0,2).\]

By the Chinese Remainder Theorem, there exists an integer $a'$ satisfying both (\ref{eq:chinois1}) and 
\[a'\equiv m_0\mod q_0^3\]
We can choose it such that $R_i(t)>0$ for all $1\leq i\leq r$.

By using either Corollary \ref{sflcrible} or Corollary \ref{sfcrible} (according to whether or not $f=1$) 
on \[\text{($f$,) $P_o$, $S'$, $T'$ and $a'$}\] we obtain $\mathscr{F}_2$, a set of infinitely many $t\in\Z$ such that \[P_o(t)=q_0^2l_{t},\quad(\text{ and }\lambda(f(t))=+1,)\]
where $l$ is a squarefree integer coprime to every element of $S'$ and where $q_0^{-6}c_6(t)$ is a square modulo $q_0$. As previously, we choose $\mathscr{F}_2$ such that $\forall t\in\mathscr{F}_2$ one has $R_i(t)>0$ for all $1\leq i\leq r$ -- except in case $P_o$ is insipid: for all $i\not=i_0$, with the exception that $R_{i_o}(t)<0$. By Proposition \ref{variation3}, all the elements of $\mathscr{F}_2$ are such that their fibers on $\E$ have the same root number.

To end the proof, we use Proposition \ref{variation3} part A to show that for all $t_1\in\mathscr{F}_1$, and every $t_2\in\mathscr{F}_2$, one has \[W\left(\E_{t_1}\right)=-W\left(\E_{t_2}\right).\]
\end{proof}

\begin{rem}{An alternative proof for the case where $\E$ has multiplicative reduction can be found in Manduchi's article where she proves the following:}

\begin{thm}\label{thmmult}\cite[Theorem 2]{Manduchi}
Let $\E$ be an non-isotrivial family of elliptic curves satisfying the hypotheses of Corollary \ref{corogeneral} with a linear place of multiplicative reduction on $\E$ outside of $-\deg$. 
Then the sets $W_\pm(\E,\Z)$ are both infinite. 
\end{thm}

In her proof, she focuses rather on making the Liouville function vary.
\end{rem}

\section{Families with coefficients of bounded degrees}\label{sectionexample}

Let $\E$ be a family of elliptic curves given for $t\in\Q$ by the Weierstrass equation 
$$\E_t:y^2=x^3+a_2(t)x^2+a_4(t)x+a_6(t)$$
where $\deg a_i\leq 2$ for $i=2,4,6$.
Suppose that there is no place of multiplicative reduction other than possibly $-\deg$ (or in other words \textit{potentially parity-biaised}), then Bettin, David and Delauney prove in recent work \cite[Theorem 7 and 8]{BDD} that there are essentially 6 different classes of non-isotrivial such famillies, namely:
\begin{align}
\mathcal{F}_s(t):& y^2=x^3+3tx^2+3sx+st\text{, with }s\in\Z_{\not=0};\\
\mathcal{G}_w(t):& wy^2=x^3+3tx^2+3tx+t^2\text{, with }w\in\Z_{\not=0}; \\
\mathcal{H}_w(t):&wy^2=x^3+(8t^2-7t+3)x^2-3(2t-1)x+(t+1)\text{, with }w\in\Z_{\not=0};\\
\mathcal{I}_w(t):&wy^2=x^3+t(t-7)x^2-6t(t-6)x+2t(5t-27)\text{, with }w\in\Z_{\not=0};\\
\mathcal{J}_{m,w}(t):&wy^2=x^3+3t^2x^2-3mtx+m^2\text{, with }m,w\in\Z_{\not=0};\\
\mathcal{L}_{w,s,v}(t):&wy^2=x^3+3(t^2+v)x^2+3sx+s(t^2+v)\text{, with }s\in\Z_{\not=0}.
\end{align}

Observe that those families have no place of multiplicative reduction except possibly at $-deg$. The bad places of those surfaces are:
\begin{description}
\item[For $\mathcal{F}_s$] $t^2-s$ ($II$), $-deg$ ($I_8^*$);
\item[For $\mathcal{G}_w$] $t-1$ ($II$), $t$ ($III$), $-deg$ ($I_7^*$);
\item[For $\mathcal{H}_w$] $t$ ($III$), $t^2-11/8t+1$ ($II$), $-deg$ ($I_5$);
\item[For $\mathcal{I}_w$] $t$ ($II$), $t^2-10t+27$ ($III$), $-deg$ ($I_4$);
\item[For $\mathcal{J}_{m,w}$] $t^3+m$ ($II$), $-deg$ ($I_6$);
\item[For $\mathcal{L}_{w,s,v}$] $t^4+2vt^2+v^2-s$ ($II$), $-deg$ ($I_4$).
\end{description}

As a corollary of Theorem \ref{thmgeneral}, any family isomorphic to one of the form $\mathcal{G}_w$, $\mathcal{H}_w$, $\mathcal{I}_w$ and $\mathcal{J}_{m,w}$ have infinitely many integer fibers with negative (resp. positive) root number. Hence, the only surfaces among those six families to which our Theorem \ref{thmgeneral} do not apply are the only of the form $\mathcal{F}_s$ with $s=-3s^2$ and $\mathcal{L}$ in very special circonstances.

In \cite{Chinis}, Chinis computes the average root number on the families $\mathcal{F}_s$, and shows that $\mathcal{F}_s$ is \emph{parity biased} over $\Z$ (i.e. the average root number over $\Z$ is not 0) if and only if $s\not\equiv1,3,5\mod 8$.
The \emph{average root number} of a family of elliptic curve $\E$ over $\Z$ is defined as $$Av_\Z(W(\E)):=\lim_{T\rightarrow\infty}{\frac{1}{2T}\sum_{\vert t \vert\leq T}{W(\E_t)}},$$ provided that the limit exists. However, even knowing that $\mid Av_\Z(W(\E))\mid=1$ would not be enough to determine whether or not the root number is constant. A work in progress by the author and Rena Chu aims at finding conditions on $s,a,b\in\Z$ for $\mathcal{F}_s(at+b)$ to have constant root number \cite{Rena}.

\bibliographystyle{alpha}
\bibliography{biblio}

\begin{thebibliography}{BDD18}

\bibitem[BDD18]{BDD}
S.~Bettin, C.~David, and C.~Delaunay.
\newblock Non-isotrivial elliptic surfaces with non-zero average root number.
\newblock {\em Journal of Number Theory}, 191:1--84, 2018.

\bibitem[CCH05]{HCC}
B.~Conrad, K.~Conrad, and H.~Helfgott.
\newblock Root numbers and ranks in positive characteristic.
\newblock {\em Adv. Math.}, 198(2):684--731, 2005.

\bibitem[CD20]{Rena}
R.~Chu and J.~Desjardins.
\newblock Constant root number on integer fibres of elliptic surfaces.
\newblock arXiv:2011.02386, 2020.

\bibitem[Chi19]{Chinis}
J.~Chinis.
\newblock {Computing the average root number of an elliptic surface}.
\newblock {\em Journal of Number Theory}, 194:83--116, 2019.

\bibitem[Del73]{Del}
P.~Deligne.
\newblock Les constantes des {\'e}quations fonctionnelles des fonctions {$L$}.
\newblock In {\em Modular functions of one variable, {II} ({P}roc. {I}nternat.
  {S}ummer {S}chool, {U}niv. {A}ntwerp, {A}ntwerp, 1972)}, pages 501--597.
  Lecture Notes in Math., Vol. 349. Springer, Berlin, 1973.

\bibitem[Des16]{Desjardinsthese}
J.~Desjardins.
\newblock {\em Densit{\'e} des points rationnels sur les surfaces elliptiques
  et les surfaces de del Pezzo de degr{\'e} 1}.
\newblock PhD thesis, Universit{\'e} Paris-Diderot - Paris VII, November 2016.

\bibitem[Des18]{Desjardins1}
J.~Desjardins.
\newblock On the variation of the root number of the fibers of families of
  elliptic curves.
\newblock {\em Journal of the London Mathematical Society}, 99(2), 2018.

\bibitem[DN19]{JulieBartosz}
J.~Desjardins and B.~Naskr{\c e}cki.
\newblock Geometry of the {D}el {P}ezzo surface $y^2=x^3+{A}m^6+{B}n^6$.
\newblock arXiv:1911.02684, 2019.

\bibitem[Duq01]{Duquesne}
S.~Duquesne.
\newblock Integral points on elliptic curves defined by simplest cubic fields.
\newblock {\em Experiment Math.}, 10(1):91--102, 2001.

\bibitem[GM91]{GM}
F.~Gouv\^ea and B.~Mazur.
\newblock The square-free sieve and the rank of elliptic curves.
\newblock {\em Journal of the American Mathematical Society}, 4(1):1--23,
  January 1991.

\bibitem[Hal98]{Halb}
E.~Halberstadt.
\newblock Signes locaux des courbes elliptiques en 2 et 3.
\newblock {\em C. R. Acad. Sci. Paris S{\'e}r. I Math.}, 326(9):1047--1052,
  1998.

\bibitem[Hel03]{Helfgott}
H.~A. Helfgott.
\newblock On the behaviour of root numbers in families of elliptic curves.
\newblock arXiv:math/0408141v3, 2003.

\bibitem[Hel05]{HelfChowla}
H.~A. Helfgott.
\newblock The parity problem for irreducible cubic forms.
\newblock arXiv:math/0501177, 2005.

\bibitem[Hel06]{HelfChowla2}
H.~A. Helfgott.
\newblock The parity problem for reducible cubic forms.
\newblock {\em J. London Math. Soc. (2)}, 73(2):415--435, 2006.

\bibitem[Hoo67]{Hool}
C.~Hooley.
\newblock On the power free values of polynomials.
\newblock {\em Mathematika}, 14:21--26, 1967.

\bibitem[Lac18]{Lachand}
Armand Lachand.
\newblock Fonctions arithm{\'e}tiques et formes binaires irr{\'e}ductibles de
  degr{\'e} 3.
\newblock {\em Annales de l'Institut Fourier}, 68(3):1297 -- 1363, 2018.

\bibitem[Man95]{Manduchi}
E.~Manduchi.
\newblock Root numbers of fibers of elliptic surfaces.
\newblock {\em Compositio Math.}, 99(1):33--58, 1995.

\bibitem[Riz03]{Rizz}
O.~G. Rizzo.
\newblock Average root numbers for a nonconstant family of elliptic curves.
\newblock {\em Compositio Math.}, 136(1):1--23, 2003.

\bibitem[Roh93]{Rohr}
D.~E. Rohrlich.
\newblock Variation of the root number in families of elliptic curves.
\newblock {\em Compositio Math.}, 87(2):119--151, 1993.

\bibitem[Tat75]{Tatefibre}
J.~Tate.
\newblock Algorithm for determining the type of a singular fibre in an elliptic
  pencil.
\newblock {\em Lect. Notes in Math.}, Modular functions of one variable IV
  (Antwerpen 1972)(476):33--52, 1975.

\bibitem[VA11]{VA}
A.~V{\'a}rilly-Alvarado.
\newblock Density of rational points on isotrivial rational elliptic surfaces.
\newblock {\em Algebra \& Number Theory}, 5:659--690, 2011.

\bibitem[Was87]{Washington}
L.~C. Washington.
\newblock Class numbers of the simplest cubic fields.
\newblock {\em Math. Comp.}, 48(1):371--384, 1987.

\end{thebibliography}

%\end{otherlanguage}
\end{document}